\documentclass[12pt,twoside,reqno]{amsart}
\usepackage{amsmath}
\usepackage{amsfonts}
\usepackage{amssymb}
\usepackage{color}
\usepackage{mathrsfs}
\usepackage{cite}
\newcommand{\beqn}{\begin{equation}}
\newcommand{\eeqn}{\end{equation}}
\newcommand{\bean}{\begin{eqnarray}}
\newcommand{\eean}{\end{eqnarray}}
\DeclareMathAlphabet{\mathpzc}{OT1}{pzc}{m}{it}
\newtheorem{theorem}{Theorem}[section]

\newtheorem{lemma}[theorem]{Lemma}

\numberwithin{equation}{section}
\newcommand{\e}{\varepsilon}
\begin{document}
\title{Finite time singularity in a MEMS model revisited}
\thanks{Partially supported by the French-German PROCOPE project 30718Z}

\author{Philippe Lauren\c{c}ot}
\address{Institut de Math\'ematiques de Toulouse, UMR~5219, Universit\'e de Toulouse, CNRS \\ F--31062 Toulouse Cedex 9, France}
\email{laurenco@math.univ-toulouse.fr}

\author{Christoph Walker}
\address{Leibniz Universit\"at Hannover\\ Institut f\" ur Angewandte Mathematik \\ Welfengarten 1 \\ D--30167 Hannover\\ Germany}
\email{walker@ifam.uni-hannover.de}


\date{\today}

\begin{abstract}
A free boundary problem modeling a microelectromechanical system (MEMS) consisting of a fixed ground plate and a deformable top plate is considered, the plates being held at different electrostatic potentials. It couples a second order semilinear parabolic equation for the deformation of the top plate to a Laplace equation for the electrostatic potential in the device. The validity of the model is expected to break down in finite time when the applied voltage exceeds a certain value, a finite time singularity occurring then. This result, already known for non-positive initial configurations of the top plate, is here proved for arbitrary ones and thus now includes, in particular, snap-through instabilities.
\end{abstract}

\maketitle

%
%
\pagestyle{myheadings}
\markboth{\sc{Ph.~Lauren\c cot \& Ch.~Walker}}{\sc{Finite time singularity in a MEMS model}}

\section{Introduction}\label{sec1}

An important feature of microelectromechanical (MEMS) devices is the so-called \textit{pull-in} instability which occurs in principle when the potential difference applied across the device exceeds a certain threshold value \cite{Ja12, PB03}. For an idealized electrostatic MEMS actuator consisting of an elastic plate coated with a thin dielectric film held at potential one (after normalization) and suspended above a rigid conducting ground plate held at potential zero, this phenomenon corresponds to the touchdown of the elastic plate on the ground plate. It is of utmost importance to figure out whether and when it does take place. Indeed, though pull-in might be a sought-for behavior (in switching applications for instance) or not (for micro-mirrors, in particular), its possible occurrence anyway has a strong influence on the operating conditions of the MEMS device. From a mathematical viewpoint, touchdown takes place when the vertical deflection of the elastic plate goes beyond a certain value at some time $T_*$. More precisely, we assume that the ground and elastic plates have the same shape $D$, which is a smooth bounded domain of $\mathbb{R}^n$, $n=1,2$. The ground plate is assumed to be located at height $z=-1$ (so that it corresponds to the surface $D\times\{-1\}$ in $\mathbb{R}^{n+1}$) while the elastic plate at time $t\ge 0$ is the surface
$$
\{ (x,z)\in D\times \mathbb{R}\ ;\  z  = u(t,x) \} \subset \mathbb{R}^{n+1}\ ,
$$
where $u(t,x)$ denotes the deflection of the elastic plate in the vertical direction at position $x\in D$ and time $t$. The touchdown phenomenon then takes place at time $T_*>0$ if
\begin{equation}
\lim_{t\to T_*} \min_{x\in \bar{D}}\{u(t,x)\} = -1\ . \label{i1}
\end{equation}
It is actually rather well-understood for the so-called \textit{vanishing aspect ratio model}
\begin{equation}
\begin{split}
\partial_t u - \Delta u = - \frac{\lambda}{(1+u)^2}\ , & \qquad t>0\ , \ x\in D\ , \\
u = 0\ , & \qquad t>0\ , \ x\in \partial D\ , \\
u(0) = u^0\ , & \qquad x\in D\ ,
\end{split} \label{i2}
\end{equation}
where the parameter $\lambda$ is proportional to the square of the potential difference applied across the device before rescaling, and the initial condition satisfies $u^0>-1$ in $D$. A threshold value $\lambda_s$ of the parameter $\lambda$ is found for which there is no stationary solution to \eqref{i2} when $\lambda$ exceeds $\lambda_s$ while there is at least one stable stationary solution for $\lambda\in (0,\lambda_s)$. Similarly, for the evolution equation \eqref{i2}, there is a threshold value $\lambda_e(u^0)>0$ depending on the initial condition $u^0$ with the following properties: if $\lambda\in (0,\lambda_e(u^0))$, then there is a unique classical solution to \eqref{i2} which exists for all times and is well-separated from $-1$ on each finite time interval. On  the contrary, if $\lambda>\lambda_e(u^0)$, then the unique classical solution to \eqref{i2} exists only on a finite time interval $[0,T_*)$ and touchdown occurs at time $T_*$ as described in \eqref{i1}. Several results are also available for variants of \eqref{i2} where $\partial_t u$ is replaced by $\partial_t^2 u$ and/or $\Delta u$ is replaced by $\Delta^2 u$ with either clamped ($u=\partial_\nu u = 0$ on $D$) or pinned ($u=\Delta u = 0$ on $D$) boundary conditions. The results, though, are less complete and several gray areas persist. We refer to \cite{EGG10, FMPS07, GPW05, KLNT15, LWBible, LLZ14, LL12, Pe02} and the references therein for a more complete description of the available results.

Far less is known for a more complex and more precise model for MEMS devices, which describes not only the dynamics of the deflection in the vertical direction $u$, but also that of the electrostatic potential $\psi_u$ between the two plates. When $D:=(-1,1)$,  it reads
\begin{subequations}\label{u}
\begin{equation}\label{equ}
\partial_t u - \partial_x^2 u = - \lambda\, g(u)\ ,\quad x\in D\ ,\qquad t>0\ ,
\end{equation}
with clamped boundary conditions
\begin{equation}\label{bcu}
u(t,\pm 1)=0\ ,\quad t>0\ ,
\end{equation}
and initial condition
\begin{equation}\label{icu}
u(0,x)=u^0(x)\ ,\quad x\in D\ ,
\end{equation}
the electrostatic force $g(u)$ being given by
\begin{equation}
g(u(t))(x) := \varepsilon^2\ |\partial_x\psi_u(t,x,u(t,x))|^2 + |\partial_z\psi_u(t,x,u(t,x))|^2\ , \qquad t>0\ , \ x\in D\ . \label{gu}
\end{equation}
\end{subequations}
\begin{subequations}\label{psi}
The dimensionless electrostatic potential $\psi_u=\psi_u(t,x,z)$ is defined in the region
$$
\Omega(u(t)) := \left\{ (x,z)\in D\times (-1,\infty)\ :\ -1 < z < u(t,x) \right\}\,
$$
between the ground plate and the elastic plate and satisfies a rescaled Laplace's equation
\begin{equation}\label{eqpsi}
\varepsilon^2\partial_x^2\psi_u + \partial_z^2\psi_u =0\ ,\quad (x,z)\in \Omega(u(t))\ ,\quad t>0\ ,
\end{equation}
the parameter $\varepsilon>0$ being the so-called \textit{aspect ratio} of the device, that is, the ratio between the vertical and horizontal directions. Laplace's equation is supplemented with non-homogeneous boundary conditions
\begin{equation}\label{bcpsi}
\psi_u(t,x,z)=\frac{1+z}{1+u(t,x)}\ ,\quad (x,z)\in  \partial\Omega(u(t))\ ,\quad t>0\ .
\end{equation}
\end{subequations}
In particular, $\psi_u(t,x,u(t,x))=1$ and $\psi_u(t,x,-1)=0$ for $x\in D$ as required. The equation~\eqref{i2} for $n=1$ is formally derived from \eqref{u}, \eqref{psi} by setting $\e=0$ since then the solution to \eqref{psi} is given explicitly by \eqref{bcpsi} for all $x\in \Omega(u(t))$ and $t>0$. In contrast to \eqref{i2}, the initial-boundary value problem \eqref{u}, \eqref{psi} features a rather intricate coupling between the two unknowns $u$ and $\psi_u$. Indeed, the source term in \eqref{equ} governing the evolution of $u$ is proportional to the square of the trace of the gradient of $\psi_u$ on the elastic plate. The electrostatic potential, in turn, solves an elliptic equation which involves a non-smooth domain varying with $u$. Thus, the source term in \eqref{equ} depends in a nonlocal and nonlinear way on $u$. Still, it is shown in \cite[Theorem~1]{ELW14} that \eqref{u}, \eqref{psi} is locally well-posed in 
$$
S_q^2(D) := \left\{v\in W_{q}^2(D)\,;\, v(\pm 1)=0  \;\text{ and }\; -1< v(x) \text{ for } x\in D \right\}\ ,
$$  
when $q>2$. It was further reported in \cite[Theorem~2]{ELW14} that, under the additional assumption that $u^0$ is non-positive, the corresponding solution to \eqref{u}, \eqref{psi} cannot exist globally if $\lambda$ exceeds some value $\lambda_e(u^0,\varepsilon)$ depending on $u^0$ and $\varepsilon$: indeed, for $\lambda>\lambda_e(u^0,\varepsilon)$ there is $T_m\in (0,\infty)$ depending on $\lambda$, $u^0$, and $\varepsilon$ such that
\begin{equation}
\lim_{t\to T_m} \min_{x\in [-1,1]}\{u(t,x)\} = -1 \quad \text{ or }\quad \lim_{t\to T_m} \|u(t)\|_{W_q^2} = \infty\, . \label{i3}
\end{equation}
According to the discussion above, only the first singularity in \eqref{i3} corresponds to the pull-in instability (recall \eqref{i1}), and this is the sole expected to occur. Though we have been unable to rule out the blowup of the Sobolev norm up to now, we point out that the non-positivity of the initial condition is important in the proof of \eqref{i3} as it entails the upper bound $u(t)\le 0$ as long as it exists, thanks to the non-positivity of the right-hand side of \eqref{equ} and the comparison principle. The purpose of this note is to prove that the non-positivity assumption on $u^0$ can be relaxed and improved to the weaker assumption that $u^0$ is simply bounded from above. It is worth mentioning that relaxing the sign condition on $u^0$ is not just a mere mathematical improvement, but is physically relevant, for instance, in the study of the so-called \textit{snap-through} instability, where the shape of the plate is initially an arch (such as $u^0(x) = h(1+\cos(\pi x))$ for $x\in D$) or a bell \cite{DB09, KISSC08, OY10}. Also, non-positivity plays a crucial role in the proof of the occurrence of a finite time singularity in a related MEMS model \cite{Li16}. More precisely, we prove the following result.

\begin{theorem}\label{th1}
Let $\varepsilon>0$ and $q>2$ and consider $u^0\in S_q^2(D)$. There is $\lambda^*(u^0,\varepsilon)>0$ depending on $\varepsilon$ and $u^0$ such that, if $\lambda>\lambda^*$, then the maximal existence time $T_m$ of the corresponding solution $(u,\psi_u)$ to \eqref{u}, \eqref{psi} is finite and
$$
\lim_{t\to T_m} \min_{x\in [-1,1]}\{u(t,x)\} = -1 \;\;\text{ or }\;\; \lim_{t\to T_m} \|u(t)\|_{W_q^2} = \infty\ . 
$$
\end{theorem}

A well-known technique to prove the occurrence of a finite time singularity in evolution equations is the so-called eigenfunction method \cite[Theorem~8]{Ka63}. Owing to the nonlocality of the right-hand side of \eqref{equ}, a direct application of this technique to \eqref{u}, \eqref{psi} seems to fail. Nevertheless, we developed in \cite{ELW14} a nonlinear version of this technique which allowed us therein to prove Theorem~\ref{th1} under the additional assumption that $u^0$ is non-positive, the latter being used in particular to control the supplementary nonlinear terms. Not surprisingly, the proof of Theorem~\ref{th1} herein relies on the same technique and borrows some steps of the proof of \cite[Theorem~2]{ELW14}. Besides some technical variations, the main new step in the proof is to use the non-positivity of the right hand side of \eqref{equ} and the decay properties of the linear heat equation to control in a suitable way the contribution of the positive part of $u$, see Lemma~\ref{lem2} below.

\section{Finite time singularity}

Let $\varepsilon>0$, $\lambda>0$, $q>2$, and consider $u^0\in S_q^2(D)$. Denoting the corresponding solution to \eqref{u}, \eqref{psi} by $(u,\psi_u)$ and its maximal existence time by $T_m$, we recall from \cite[Proposition~5]{ELW14} that 
\begin{equation}
u(t)\in S_q^2(D)\ , \qquad t\in [0,T_m)\ , \label{p0}
\end{equation}
and
$\psi_u(t)\in H^2(\Omega(u(t))$ for all $t\in [0,T_m)$. Consequently, 
$$
\gamma_m(t,x) := \partial_z \psi_u(t,x,u(t,x))\ , \qquad (t,x)\in [0,T_m) \times D\ ,
$$
is well-defined and, since 
$$
\partial_x \psi_u(t,x,u(t,x)) = - \partial_x u(t,x) \gamma_m(t,x)\ , \qquad (t,x)\in [0,T_m) \times D\ ,
$$
by \eqref{bcpsi}, the right-hand side of \eqref{equ} also reads
$$
- \lambda g(u) = - \lambda (1+\varepsilon^2 |\partial_x u|^2) \gamma_m^2\ .
$$
Introducing the principal eigenvalue $\mu_1 := \pi^2/4$ of the Laplace operator $-\partial_x^2$ in $L_2(D)$ with homogeneous Dirichlet boundary conditions and the corresponding eigenfunction $\zeta_1(x) := \pi \cos(\pi x/2)/4$, $x\in D$, we define 
$$
E_\alpha(t) := \int_D \zeta_1(x) \left( u(t,x) + \frac{\alpha}{2} u(t,x)^2 \right)\ \mathrm{d}x\ , \qquad t\in [0,T_m)\ ,
$$
where $\alpha\in [0,1]$ is to be determined later on. Observe that \eqref{p0} and the properties of $\zeta_1$ guarantee that
\begin{equation}
E_\alpha(t) > -1\ , \qquad t\in [0,T_m)\,,\label{p00}
\end{equation}
for all $\alpha\in [0,1]$. We shall show that for any sufficiently large value of $\lambda$ we can choose $\alpha\in [0,1]$ such that \eqref{p00} can only hold true when $T_m<\infty$. This then implies Theorem~\ref{th1}.

\medskip

The starting point of the analysis is the following upper bound on $E_\alpha$.

\begin{lemma}\label{lem2}
Let $\alpha\in [0,1]$ with $\alpha\le 2/(1+ (\max u^0)_+)$. There is a positive constant $C_0$ depending only on $\max u^0$ such that
\begin{equation}
E_\alpha(t) \le C_0 e^{-\mu_1 t}\ , \qquad t\in [0,T_m)\ . \label{p1}
\end{equation}
\end{lemma}

\begin{proof}
Let $v$ and $w$ be the solutions to 
\begin{eqnarray}
\partial_t v - \partial_x^2 v = 0\ , \qquad (t,x)\in (0,T_m)\times D\ , & \qquad v(t,\pm1)=0\ , \qquad t\in (0,T_m)\ , \label{p2} \\
\partial_t w - \partial_x^2 w = - \lambda g(u)\ , \qquad (t,x)\in (0,T_m)\times D\ , & \qquad w(t,\pm1)=0\ , \qquad t\in (0,T_m)\ , \label{p3}
\end{eqnarray}
with initial conditions
\begin{equation}
v(0,x) = \max\{ u^0(x), 0\}\ , \quad w(0,x) = u^0 - \max\{ u^0(x), 0\}\ , \qquad x\in D\ . \label{p4}
\end{equation}
Since $w(0,x) \le 0 \le v(0,x) \le (\max u^0)_+$ for $x\in D$, we infer from \eqref{p2}, \eqref{p3}, \eqref{p4}, and the comparison principle that
$$
w(t,x) \le 0 \le v(t,x) \le (\max u^0)_+\ , \qquad (t,x)\in (0,T_m)\times D\ , 
$$
and
$$
\|v(t)\|_2 \le e^{-\mu_1 t} \|v(0)\|_2 \le \sqrt{2}  (\max u^0)_+ e^{-\mu_1 t} \ , \qquad t\in (0,T_m)\ .
$$
Also,
$$
w(t,x) = u(t,x) - v(t,x) \ge -1 - (\max u^0)_+\ , \qquad (t,x)\in (0,T_m)\times D\ ,
$$
so that
$$
1 + \frac{\alpha}{2} w(t,x) \ge 1 - \frac{\alpha}{2} \left( 1 + (\max u^0)_+ \right) \ge 0\ , \qquad (t,x)\in (0,T_m)\times D\ ,
$$
thanks to the constraint on $\alpha$. Therefore, for $t\in (0,T_m)$,
\begin{align*}
E_\alpha(t) & = \int_D \zeta_1 \left[ v(t) + \frac{\alpha}{2} v(t)^2 + \alpha v(t) w(t) + w(t) \left( 1 + \frac{\alpha}{2} w(t) \right) \right]\ \mathrm{d}x \\
& \le \int_D \zeta_1 \left[ v(t) + \frac{\alpha}{2} v(t)^2 \right]\ \mathrm{d}x  \le \frac{\pi}{4} \left( \sqrt{2} \|v(t)\|_2 + \frac{\|v(t)\|_2^2}{2} \right) \\
& \le \pi \left[ (\max u^0)_+ + (\max u^0)_+^2 \right] e^{-\mu_1 t}\ ,
\end{align*}
and the proof is complete.
\end{proof}

We next derive a differential inequality for $E_\alpha$. Though the proof is quite similar to that performed in \cite{ELW14}, we nevertheless recall it here not only for the sake of completeness but also to shed some light on the importance of choosing $\alpha>0$. 

\begin{lemma}\label{lem3}
Let $p\ge 1$ and $\delta>0$. If 
\begin{equation}
\alpha := \frac{\lambda \varepsilon^2}{\lambda \varepsilon^2 + 4 \delta^2}\in (0,1)\ , \label{p11}
\end{equation}
then
\begin{equation}
\frac{\mathrm{d}}{\mathrm{d}t} E_\alpha \le \mathcal{F}_{p,\delta}(E_\alpha)\ , \qquad t\in [0,T_m)\ , \label{p5}
\end{equation}
where
\begin{equation}
\mathcal{F}_{p,\delta}(y) := \mu_1 +  \frac{4\delta\lambda}{p(\lambda \varepsilon^2 + 4 \delta^2)} \left[ \frac{\mu_1 \varepsilon^2}{p} + \frac{p}{4\delta} +  \frac{p \mu_1 \varepsilon^2}{p+1} y - \frac{1}{1+y} \right]\ , \qquad y>-1\ . \label{p6}
\end{equation}
\end{lemma}

\begin{proof}
It readily follows from \eqref{equ} that, for $t\in [0,T_m)$,
\begin{align*}
\frac{\mathrm{d}}{\mathrm{d}t} E_\alpha & = \int_D \zeta_1 (1+\alpha u) (\partial_x^2 u - \lambda g(u))\ \mathrm{d}x \\
& = - \mu_1 E_\alpha - \alpha \int_D \zeta_1 |\partial_x u|^2\ \mathrm{d}x - \lambda \int_D \zeta_1 (1+\alpha u) g(u)\ \mathrm{d}x\ . 
\end{align*}
Since $\zeta_1\ge 0$ and $1+\alpha u\ge 1-\alpha$ in $D$, we conclude that
\begin{equation}
\frac{\mathrm{d}}{\mathrm{d}t} E_\alpha + \mu_1 E_\alpha + \alpha \int_D \zeta_1 |\partial_x u|^2\ \mathrm{d}x \le - \lambda (1-\alpha) \int_D \zeta_1 g(u) \, \mathrm{d}x\ , \qquad t\in [0,T_m)\ . \label{p7}
\end{equation}

The next step is to estimate the last term on the right hand side of \eqref{p7} in terms of $E_\alpha$. To this end, we recall the following results established in \cite[Lemma~9 \&\ Lemma~10]{ELW14} which do not rely on the sign of $u^0$. For $t\in [0,T_m)$ and $p\in [1,\infty)$, we have
\begin{equation}
\frac{4p}{(p+1)^2}\ \int_D \frac{\zeta_1(x)}{1+u(t,x)}\ \mathrm{d}x \le p\ \int_{\Omega(u(t))} \zeta_1(x)\ \psi_u(t,x,z)^{p-1}\ |\partial_z \psi_{u(t)}(t,x,z)|^2\ \mathrm{d}(x,z) \label{p8}
\end{equation}
and
\begin{align}
\int_D &\zeta_1(x)\ \left( 1 + \varepsilon^2\ |\partial_x u(t,x)|^2 \right)\ \gamma_m(t,x)\ \mathrm{d}x \nonumber\\
& = p \int_{\Omega(u(t))} \zeta_1(x)\ \psi_u(t,x,z)^{p-1} \left[ \varepsilon^2\ |\partial_x \psi_u(t,x,z)|^2 + |\partial_z \psi_u(t,x,z)|^2 \right] \mathrm{d}(x,z) \nonumber\\
& \quad + \frac{\mu_1 \varepsilon^2}{p+1} \int_{\Omega(u(t))} \zeta_1(x) \psi_u(t,x,z)^{p+1} \mathrm{d}(x,z) \nonumber\\ 
&  \quad -\ \frac{\mu_1 \varepsilon^2}{(p+1)(p+2)} - \frac{\mu_1\,\varepsilon^2}{p+1}\ \int_D \zeta_1(x)\ u(t,x)\ \mathrm{d}x\ . \label{p9}
\end{align}
Since $\psi_u\ge 0$ by the comparison principle, it follows from \eqref{p8}, \eqref{p9}, and the non-negativity of $\zeta_1$ that, for $t\in [0,T_m)$,
\begin{align*}
\int_D \zeta_1 (1+\varepsilon^2 |\partial_x u|^2) \gamma_m\ \mathrm{d}x & \ge p \int_{\Omega(u(t))} \zeta_1 \psi_u ^{p-1} |\partial_z \psi_u|^2 \mathrm{d}(x,z) \\
& \quad - \frac{\mu_1 \varepsilon^2}{(p+1)(p+2)} - \frac{\mu_1\,\varepsilon^2}{p+1}\ \int_D \zeta_1 u\ \mathrm{d}x \\ 
& \ge \frac{4p}{(p+1)^2} \int_D \frac{\zeta_1}{1+u}\ \mathrm{d}x - \frac{\mu_1 \varepsilon^2}{p^2} - \frac{\mu_1 \varepsilon^2}{p+1} E_\alpha \\
& \ge \frac{4p}{(p+1)^2} \int_D \frac{\zeta_1}{1+u+ \alpha u^2/2}\ \mathrm{d}x - \frac{\mu_1 \varepsilon^2}{p^2} - \frac{\mu_1 \varepsilon^2}{p+1} E_\alpha \\
& \ge \frac{1}{p} \int_D \frac{\zeta_1}{1+u+ \alpha u^2/2}\ \mathrm{d}x - \frac{\mu_1 \varepsilon^2}{p^2} - \frac{\mu_1 \varepsilon^2}{p+1} E_\alpha \ .
\end{align*}
Using Jensen's inequality we end up with
\begin{equation}
\int_D \zeta_1 (1+\varepsilon^2 |\partial_x u|^2) \gamma_m\ \mathrm{d}x \ge \frac{1}{p} \frac{1}{1+E_\alpha} - \frac{\mu_1 \varepsilon^2}{p^2} - \frac{\mu_1 \varepsilon^2}{p+1} E_\alpha \ . \label{p10}
\end{equation}
We now deduce from Young's inequality that, for any $\delta>0$, 
$$
\int_D \zeta_1 (1+\varepsilon^2 |\partial_x u|^2) \gamma_m\ \mathrm{d}x \le \delta \int_D \zeta_1 (1+\varepsilon^2 |\partial_x u|^2) \gamma_m^2\ \mathrm{d}x + \frac{1}{4\delta} \int_D \zeta_1 (1+\varepsilon^2 |\partial_x u|^2)\ \mathrm{d}x\ .
$$
Therefore, by \eqref{p10},
\begin{align*}
\int_D \zeta_1 g(u)\ \mathrm{d}x & = \int_D \zeta_1 (1+\varepsilon^2 |\partial_x u|^2) \gamma_m^2\ \mathrm{d}x \\
& \ge \frac{1}{\delta} \int_D \zeta_1 (1+\varepsilon^2 |\partial_x u|^2) \gamma_m\ \mathrm{d}x - \frac{1}{4\delta^2} \int_D \zeta_1 (1+\varepsilon^2 |\partial_x u|^2)\ \mathrm{d}x \\
& \ge \frac{1}{\delta p} \frac{1}{1+E_\alpha} - \frac{\mu_1 \varepsilon^2}{\delta p^2} - \frac{\mu_1 \varepsilon^2}{\delta (p+1)} E_\alpha - \frac{1}{4\delta^2} - \frac{\varepsilon^2}{4\delta^2} \int_D \zeta_1 |\partial_x u|^2\ \mathrm{d}x\ .
\end{align*}
Combining the above inequality with \eqref{p7} gives
\begin{align*}
\frac{\mathrm{d}}{\mathrm{d}t} E_\alpha + \mu_1 E_\alpha + \alpha \int_D \zeta_1 |\partial_x u|^2\ \mathrm{d}x & \le - \frac{\lambda (1-\alpha)}{\delta p} \left[ \frac{1}{1+E_\alpha} - \frac{\mu_1 \varepsilon^2}{p} - \frac{p}{4\delta} - \frac{p \mu_1 \varepsilon^2}{p+1}  E_\alpha  \right] \\
& \quad + \frac{\lambda (1-\alpha) \varepsilon^2}{4\delta^2} \int_D \zeta_1 |\partial_x u|^2\ \mathrm{d}x \ .
\end{align*}
Owing to the choice \eqref{p11} of $\alpha$, we see that
$$
\alpha = \frac{\lambda (1-\alpha) \varepsilon^2}{4\delta^2}\ ,
$$
so that the terms involving $|\partial_x u|^2$ cancel in the previous inequality, and we end up with
$$
\frac{\mathrm{d}}{\mathrm{d}t} E_\alpha \le - \mu_1 E_\alpha + \frac{4\delta\lambda}{p(\lambda \varepsilon^2 + 4 \delta^2)} \left[ \frac{\mu_1 \varepsilon^2}{p} + \frac{p}{4\delta} + \frac{p\mu_1 \varepsilon^2}{p+1} E_\alpha - \frac{1}{1+E_\alpha} \right]\ ,
$$
hence \eqref{p5} as $E_\alpha> -1$ by \eqref{p00}.
\end{proof}

\medskip

\begin{proof}[Proof of Theorem~\ref{th1}]
Let $p\ge 1$ and $\delta>0$. We investigate the properties of the function $\mathcal{F}_{p,\delta}$ defined in \eqref{p6} and aim at finding values of the yet undetermined parameters $p\ge 1$ and $\delta>0$ which are suitable for our purpose. To begin with, we set 
$$
\alpha = \frac{\lambda \varepsilon^2}{\lambda \varepsilon^2 + 4 \delta^2}
$$
as in \eqref{p11} and require that
$$
\frac{\lambda \varepsilon^2}{\lambda \varepsilon^2 + 4 \delta^2} \le \frac{2}{1 + (\max u^0)_+}\ ,
$$
that is, 
\begin{equation}
\delta \ge \chi_\varepsilon \frac{\sqrt{\lambda}}{2}\ , \quad \chi_\varepsilon := \varepsilon \sqrt{\frac{\left( (\max u^0)_+ - 1 \right)_+}{2}}\ . \label{p12}
\end{equation}
Then, according to Lemma~\ref{lem2}, the positive part of $E_\alpha(t)$ decays rapidly as time increases, so that only the behavior of $\mathcal{F}_{p,\delta}$ in a neighborhood of the interval $(-1,0]$ is expected to matter in the following. We first notice that $\mathcal{F}_{p,\delta}$ is increasing on $(-1,\infty)$ with 
$$
\mathcal{F}_{p,\delta}(0) = \mu_1 + \frac{4\delta\lambda}{p(\lambda \varepsilon^2 + 4 \delta^2)} \left( \frac{\mu_1 \varepsilon^2}{p} + \frac{p}{4\delta} - 1 \right)\ .
$$
We now choose $p=1+2\mu_1 \varepsilon^2$ and $\delta = \chi \sqrt{\lambda}/2$ with $\chi := \max\{1 , \chi_\varepsilon\}$, the parameter $\chi_\varepsilon$ being defined in \eqref{p12}. Then
\begin{align*}
\mathcal{F}_{p,\delta}(0) & = \mu_1 + \frac{2\chi \sqrt{\lambda}}{(1+2\mu_1 \varepsilon^2)(\chi^2 + \varepsilon^2)} \left( \frac{\mu_1 \varepsilon^2}{1+2\mu_1 \varepsilon^2} + \frac{1+2\mu_1 \varepsilon^2}{2\chi \sqrt{\lambda}} - 1 \right) \\
& \le \mu_1 + \frac{2\chi \sqrt{\lambda}}{(1+2\mu_1 \varepsilon^2)(\chi^2 + \varepsilon^2)} \left( \frac{1+2\mu_1 \varepsilon^2}{2\chi \sqrt{\lambda}} - \frac{1}{2} \right) \\
& \le \mu_1 +  \frac{\chi}{(1+2\mu_1 \varepsilon^2)(\chi^2+\varepsilon^2)} \left( \frac{1+2\mu_1 \varepsilon^2}{\chi} - \sqrt{\lambda} \right) \\
& \le \frac{\chi}{(1+2\mu_1 \varepsilon^2)(\chi^2+\varepsilon^2)} \left( \frac{1+2\mu_1 \varepsilon^2}{\chi} \left( 1 + \mu_1 (\chi^2+\varepsilon^2) \right) - \sqrt{\lambda} \right)\ .
\end{align*}
Consequently, if 
\begin{equation}
\sqrt{\lambda} > \frac{1+2\mu_1 \varepsilon^2}{\chi} \left( 1 + \mu_1 (\chi^2+\varepsilon^2) \right)\ , \label{p13}
\end{equation}
then $\mathcal{F}_{p,\delta}(0)<0$ and there is $y_{p,\delta}>0$ such that \begin{equation}
\mathcal{F}_{p,\delta}(y_{p,\delta}) < 0\ . \label{p14}
\end{equation}

Now, assume for contradiction that $T_m=\infty$. According to Lemma~\ref{lem2}, there is $t_{p,\delta}>0$ such that $E_\alpha(t)\le y_{p,\delta}$ for $t\ge t_{p,\delta}$ and the monotonicity of $\mathcal{F}_{p,\delta}$ further entails that $\mathcal{F}_{p,\delta}(E_\alpha(t)) \le \mathcal{F}_{p,\delta}(y_{p,\delta})$ for $t\ge t_{p,\delta}$. We then infer from Lemma~\ref{lem3} that
$$
\frac{\mathrm{d}}{\mathrm{d}t} E_\alpha(t) \le \mathcal{F}_{p,\delta}(y_{p,\delta}) < 0\ , \qquad t\ge t_{p,\delta}\ ,
$$
hence, after integration, $E_\alpha(t) \le E_\alpha(t_{p,\delta}) + \mathcal{F}_{p,\delta}(y_{p,\delta}) (t-t_{p,\delta})$ for $t\ge t_{p,\delta}$. Thus, there is $T>t_{p,\delta}$ such that $E_\alpha(T)<-1$, which contradicts \eqref{p00}. Therefore, $T_m$ is finite and the proof of Theorem~\ref{th1} is complete.
\end{proof}

\section*{Acknowledgments}

Part of this work was done while PhL enjoyed the hospitality of the Institut f\"ur Angewandte Mathematik, Leibniz Universit\"at Hannover.



\end{document}